\documentclass[12pt]{amsart}
\renewcommand{\bar}{\overline}
\usepackage{amsmath,amsthm,amscd,euscript,longtable}
\renewcommand{\bar}{\overline}
\def \r{\mathbb R}

\def \z{\mathbb Z}

\DeclareMathOperator{\grad}{grad}
 
\DeclareMathOperator{\sgn}{sgn}

\newtheorem{theorem}{Theorem}[section]
\newtheorem{lemma}[theorem]{Lemma}

\newtheorem{proposition}[theorem]{Proposition}
\newtheorem{corollary}[theorem]{Corollary}
\theoremstyle{remark}
\newtheorem{remark}[theorem]{Remark}
\theoremstyle{definition}

\newtheorem{example}[theorem]{Example}
\newtheorem{problem}{Problem}

\title{On the flexibility of Kokotsakis meshes.}
\author{Oleg N. Karpenkov}
\date{09 December 2008}

\thanks{The work is partially supported by RFBR SS-709.2008.1
grant, by RFBR 00-01-02805-CNRSL\_a grant, and by  FWF grant
No.~S09209.}

\keywords{Meshes, flexibility, infinitesimal flexibility}

\email[Oleg N. Karpenkov]{karpenkov@tugraz.at}
\address{TU Graz /Kopernikusgasse 24, A 8010 Graz, Austria/}

\begin{document}
\input epsf

\begin{abstract}
In this paper we study geometric, algebraic, and computational
aspects of flexibility and infinitesimal flexibility of Kokotsakis
meshes. A Kokotsakis mesh is a mesh that consists of a face in
the middle and a certain band of faces attached to the middle face by
its perimeter. In particular any $(3\times 3)$-mesh made of
quadrangles is a Kokotsakis mesh. We express the infinitesimal
flexibility condition in terms of Ceva and Menelaus theorems.
Further we study semi-algebraic properties of the set of flexible
meshes and give equations describing it. For $(3\times 3)$-meshes
we obtain flexibility conditions in terms of face angles.
\end{abstract}

\maketitle

\tableofcontents

\section*{Introduction}

Consider a mesh having one special $n$-vertex face and a ``strip''
of faces attached to the special face along the
perimeter, such that any vertex of a special face is a vertex of
exactly three faces in the strip. This mesh is called a {\it
Kokotsakis mesh}. Kokotsakis meshes with polygonal faces are known
in the literature as {\it polyhedral surfaces of Kokotsakis type}.
In this paper we study conditions of flexibility and infinitesimal
flexibility for Kokotsakis meshes.

The first steps in studying such meshes were made in 1932 by
A.~Kokotsakis in his paper~\cite{Kok}, where he found conditions
of infinitesimal flexibility. Recently Kokotsakis meshes for
$4$-gons (such meshes are also called {\it $(3{\times}
3)$-meshes}) were used in the study of discrete conjugate
nets  flexibility. A.~I.~Bobenko, T.~Hoffmann, W.~K.~Schief proved
in~\cite{BHS} that a nondegenerate discrete conjugate net is isometrically
deformable if and only if all its $3{\times} 3$ subcomplexes are
isometrically deformable $(3{\times} 3)$-meshes. In~\cite{BHS}
they also showed that the conjugate nets admit infinitesimal
second-order deformations if and only if their reciprocal-parallel
surfaces constitute discrete Bianchi surfaces (see
also~\cite{Bia}).

In this paper we give a geometrical interpretation of
infinitesimal flexibility in terms of Ceva and Menelaus
configurations. Further we work with equations that define the
set of flexible Kokotsakis meshes, in particular we show
semi-algebraicity of this set. The technique developed in this
paper allows us to calculate linear independence of the gradients
of the first 6 flexibility conditions for a chosen flexible
quadrangular Kokotsakis mesh. This means that we can check if the
codimension of a stratum of meshes is smaller than or equal to 6.
The computational complexity of the problem does not allow us to
proceed further. We remind the reader that the codimension of the set of
flexible Voss meshes as well as that of symmetric flexible meshes is 8,
for information on Voss surfaces we refer to~\cite{BHS}. So it
is most probable that the codimensions of irreducible components
coincide, and equal~6, 7, or~8. If we are lucky and these
codimensions are exactly~8, than the geometric solution of the
problem can be obtained by enumeration of all possible components.

To check the conditions at a
single point we introduce more simple ``flexibility vector
calculus'' conditions. In the case of $(3{\times} 3)$-meshes we now
able to calculate the 8 first such conditions.

Finally, for the case of quadrangular Kokotsakis $(3{\times}
3)$-meshes we show that the property of the mesh to be flexible
depends only on the shapes of faces but not on the geometry of the
mesh in $\r^3$. We propose an algorithm to write these formulae
explicitly.

\vspace{2mm}

{\bf Organization of the paper.} We start in Section~1 with
showing an infinitesimal flexibility condition in algebraic and
geometric terms. Section~2 studies properties of the set of
flexible meshes. We describe equations defining this set.
In Section~3 we introduce a technique to write
flexibility conditions in terms of angles of the faces of meshes.
We conclude the paper in Section~4 with some open questions for further study.

\vspace{2mm}

{\bf General notation.} To avoid the annoying description of
multiple cases we consider the index $i$ belonging to the cyclic
group $\z/n\z$. We also choose a partial {\it ordering}
$1<2<3<\ldots<n$ for this group.

\vspace{1mm}

We denote the vertices of the special face in the middle of the
mesh by $A_i$, $i=1,\ldots,n$, preserving their order at the
boundary. Let $a_i={A_{i+1}-A_i}$. Notice that by definition we
have $\sum_{i=1}^{n}a_i=0$.

Consider a vertex $A_i$ ($i=1,\ldots, n$). We denote vertices
$V_i$ and $W_i$ such that the four faces meeting at $A_i$ contain
angles $A_{i+1}A_iA_{i-1}$, $A_{i-1}A_iV_i$, $V_iA_iW_i$, and
$W_iA_iA_{i+1}$. We denote these angles by $\alpha_i$, $\beta_i$,
$\gamma_i$, and $\varphi_i$, respectively. Let $v_i=V_i-A_i$ and
$w_i=W_i-A_i$.

\vspace{2mm}

\begin{figure}
$$\epsfbox{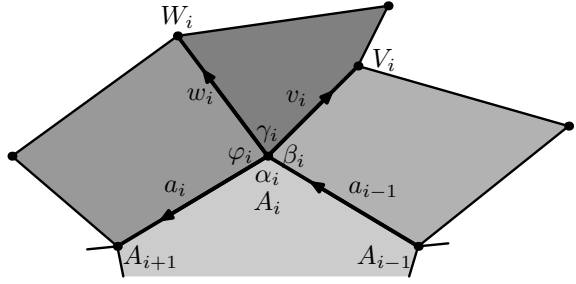}$$
\caption{Notation.}\label{part}
\end{figure}

{\it Remark.} Without loss of generality we assume that the
special face $A_1A_2\ldots A_n$ is fixed, while the other $2n$
faces may move.

\vspace{2mm}

Denote by $\langle u_1, u_2\rangle$ and by $[u_1, u_2]$ the scalar
and vector products, respectively.

\section{Infinitesimal flexibility of Kokotsakis meshes}\label{Sec1}

\subsection{Algebraic conditions}

In this subsection we recall some results on infinitesimal
flexibility from the work~\cite{Kok} by A.~Kokotsakis. We rewrite
his results in a slightly modified way more convenient for us.

Consider a $(2{\times} 2)$-mesh with a common vertex $A$. Let us
fix a face $BAC$. Let {\it BAV}, {\it VAW}, and {\it WAC} be the other three
faces. Denote by $b$, $c$, $v$, and $w$ the vectors $B-A$, $C-A$,
$V-A$, and $W-A$ respectively. This mesh is flexible in general
position and has one degree of freedom.

\begin{proposition}\label{basis}
Suppose that the vectors $v$ and $w$ are not collinear and the
vectors $b$ and $c$ are not in the plane spanned by the vectors
$v$ and $w$. Then for any infinitesimal motion $(\dot{v},\dot{w})$
of the described $(2{\times} 2)$-mesh there exists a non-zero
real $\lambda$ such that
$$
\left\{
\begin{array}{lll}
\dot{v}&=&\displaystyle  \lambda\frac{[b,v]}{(b,v,w)}\\
\dot{w}&=&\displaystyle  \lambda\frac{[c,w]}{(c,v,w)}\\
\end{array}
\right. .
$$
\end{proposition}
\begin{proof}
Any infinitesimal motion is a solution of the following linear
system with 6 variables (coordinates of $\dot{v}$ and $\dot{w}$)
and 5 equations:
$$
\left\{
\begin{array}{l}
\langle v,\dot{v}\rangle=0\\
\langle w,\dot{w}\rangle=0\\
\langle\dot{v},w\rangle+\langle v,\dot{w}\rangle=0\\
\langle b,\dot{v}\rangle=0\\
\langle c,\dot{w}\rangle=0\\
\end{array}
\right .
.
$$
On one hand the solution mentioned in the statement of the
proposition satisfies this system. On the other hand, if the
conditions on vectors $b$, $c$, $v$, and $w$ hold then the rank of
the linear system is 1 (this follows from direct calculations in
coordinates).
\end{proof}

From Proposition~\ref{basis} we immediately have the following
statement (this statement is another formulation of the result
in~\cite{Kok} by A.~Kokotsakis). For a Kokotsakis mesh $M$ put by
definition
$$
\chi(M)=\prod\limits_{i=1}^n
\frac{(a_{i-1},v_i,w_i)}{(a_i,v_i,w_i)}.
$$

\begin{corollary}\label{flex}
A Kokotsakis mesh $M$ is infinitesimally flexible if and only if
$\chi(M){=}1$. \qed
\end{corollary}

\subsection{Geometric infinitesimal condition for a Kokotsakis mesh
with planar faces}

Consider points $A_1, \ldots, A_n$ in a plane. We denote this
plane by $\pi$. Let $l_i$ denote the intersection of the plane
containing the points $A_i$, $V_i$, and $W_i$ with the plane
$\pi$. We suppose that the lines $l_i$ and $l_{i+1}$ are not
parallel, and denote their intersection point by $B_i$.

For a $(3{\times} 3)$-mesh constructed with planar faces the
infinitesimal flexibility conditions are shown in the following
proposition. For more information see the work~\cite{PW} by
H.~Pottmann, J.~Wallner and the book~\cite{Sau} by R.~Sauer.

\begin{proposition}
Consider a mesh in general position. Then it is infinitesimally
flexible if and only if the lines $B_1B_3$, $A_2,A_3$, and
$A_4A_1$ intersect in one point $($i.e. the points $A_i$, $B_i$,
$i=1,2,3,4$ are aligned in a Desargues configuration$)$.
\end{proposition}

Let us generalize this statement for an arbitrary Kokotsakis mesh
with planar polygons as faces. For $i=1,\ldots,n$ we define real
numbers $t_i$ from the following affine equations:
$$
A_i=t_iB_{i-1}+(1-t_i)B_i.
$$

\begin{theorem}\label{relations}
The closed broken line $A_1B_1A_2B_2\ldots A_nB_nA_1$ satisfies
the generalized Ceva-Menelaus condition, i.e.
$$
\prod\limits_{i=1}^{n}\frac{1-t_i}{t_i}=(-1)^n.
$$
%
\end{theorem}

\begin{remark}
If there is one point $B_i$ ``at infinity'' one should replace
$$
\frac{1-t_i}{t_i}\cdot \frac{1-t_{i+1}}{t_{i+1}} \qquad \hbox{by} \qquad
(-1)^\varepsilon \frac{1-t_{i+1}}{t_i}
$$
in the product. Here $t_i$ and $1{-}t_{i+1}$ denotes the absolute values
of the vectors $A_i{-}B_{i-1}$ and $B_{i+1}{-}A_{i+1}$ respectively. We take $\varepsilon=0$ if
the directions of these two vectors are opposite and $\varepsilon=1$ if they coincide.

We leave for the reader the other cases of some of the vertices $B_i$ being ``at infinity'' as
an exercise.
\end{remark}

\begin{proof}
Choose some orientation on the plane~$\pi$. For an ordered couple
$e_1,e_2$ of vectors denote
$$
\sgn(e_1,e_2)=\left\{
\begin{array}{cl}
1& \hbox{if $(e_1,e_2)$ is a positively oriented basis of the plane}\\
0& \hbox{if $e_1$ and $e_2$ are collinear}\\
-1& \hbox{otherwise}
\end{array}
\right. .
$$

Denote by $h_{P,QR}$ the (non-oriented) distance from the point
$P$ to the line $QR$. Then for any $i\in\{1,\ldots,n\}$ we have:
$$
\begin{array}{c}
\displaystyle \frac{(v_i,w_i,a_i)}{(v_i,w_i,a_{i-1})}=
\frac{\sgn(B_{i-1}B_i,A_iA_{i+1})h_{A_{i+1},B_{i-1}B_{i}}}{\sgn(B_{i-1}B_i,A_{i-1}A_{i})h_{A_{i+1},B_{i-1}B_{i}}}=\\
\displaystyle
\frac{\sgn(B_{i-1}B_i,A_iA_{i+1})}{\sgn(B_{i-1}B_i,A_{i-1}A_{i})}\cdot
\frac{|A_iA_{i+1}|/\sin(A_{i+1}A_iB_i)}{|A_{i-1}A_{i}|/\sin(A_{i-1}A_iB_{i-1})}.
\end{array}
$$

Therefore,
$$
\begin{array}{c}
\displaystyle \prod\limits_{i=1}^{n}\frac{1-t_i}{t_i}=
\prod\limits_{i=1}^{n}
\frac{\sgn(B_{i-2}B_{i-1},A_{i-1}A_{i})|B_{i-1}A_{i}|}{-\sgn(B_{i}B_{i+1},A_iA_{i+1})|A_iB_i|}=\\
\displaystyle
(-1)^n \cdot \prod\limits_{i=1}^{n} \left(
\frac{\sgn(B_{i-1}B_i,A_iA_{i+1})}{\sgn(B_{i-1}B_i,A_{i-1}A_{i})}\cdot
\frac{|A_iA_{i+1}|}{|A_{i-1}A_{i}|}\cdot
\frac{\sin(A_{i-1}A_iB_{i-1})}{\sin(A_{i+1}A_iB_i)} \right)=\\
\displaystyle (-1)^n \cdot
\prod\limits_{i=1}^{n}\frac{(v_i,w_i,a_i)}{(v_i,w_i,a_{i-1})}=
(-1)^n.
\end{array}
$$
The second equality holds by the law of sines for the triangles
$A_iA_{i+1}B_i$ for $i=1,\ldots,n$.
\end{proof}

Denote by $[A,B;C,D]$ the point of intersection of the
lines passing through $A,B$ and $C,D$ respectively.

Denote $P_0=B_1$. Let $O_i=[B_n,A_{i+1};B_{i+1},P_{i-1}]$, and
$P_i=[B_{i+1},B_n;B_{i},O_{i}]$, for $i=1,\ldots, n{-}2$.

\begin{proposition}\label{geom}{\bf Geometric condition I.}
The condition
$$
\prod\limits_{i=1}^{n}\frac{1-t_i}{t_i}=1
$$
is equivalent to the condition that the points $A_n$ and $P_{n-2}$
coincide $($see Figure~\ref{pentagon}.I$)$.
\end{proposition}

\begin{figure}
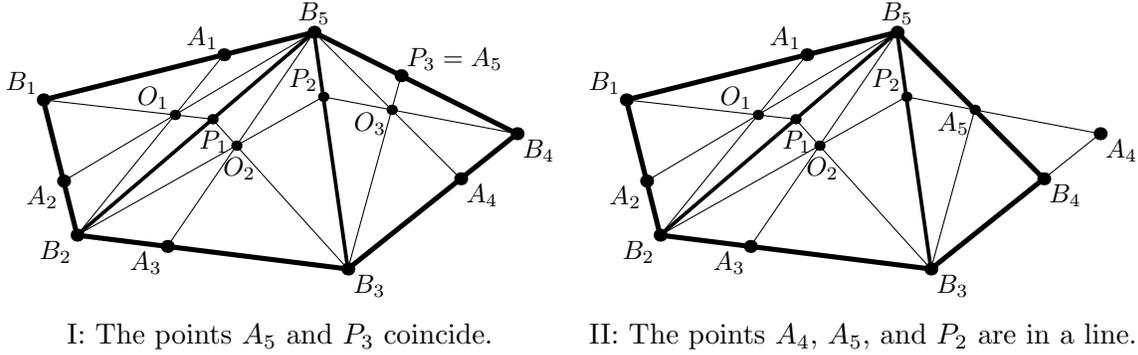

$$\epsfbox{pentagon.1} \quad \epsfbox{pentagon.2}$$
\caption{Geometric conditions I and II for a pentagon.}
\label{pentagon}
\end{figure}

In the case $n=3$ we have: {\it the lines $B_1A_3$, $B_2A_1$, and
$B_3A_2$ intersect at one point $($Ceva configuration$)$}.  In the
case $n=4$ we have: {\it the lines $B_1B_3$, $A_2,A_3$, and
$A_4A_1$ intersect at one point $($Desargues configuration$)$}.

\begin{proof}
We prove the proposition by induction on $n$.

By Ceva's theorem and Theorem~\ref{relations} the statement holds
for all configurations with $n=3$. For pairwise distinct collinear
points $A$, $B$, $C$ denote
$$
\sigma(ABC)=\left\{
\begin{array}{cl}
1& \hbox{if the point $B$ is in the segment $[A,C]$}\\
-1& \hbox{otherwise}
\end{array}
\right. .
$$

Suppose the statement holds for all configurations with $n=4$. By
Ceva's Theorem in the triangle $B_nB_1B_2$ we have:
$$
\prod\limits_{i=1}^{n}\sigma(B_{i-1}A_iBi)\frac{B_{i-1}A_i}{A_iB_i}=
\sigma(B_{n}P_1B_2)\frac{B_{n}P_1}{P_1B_2}
\prod\limits_{i=3}^{n}\sigma(B_{i-1}A_iB_i)\frac{B_{i-1}A_i}{A_iB_i}.
$$
Then by induction the statement is true for the broken line
$P_1B_2A_3B_3\ldots A_nB_nP_1$.
\end{proof}

\begin{proposition}\label{geom2}{\bf Geometric condition II.}
The condition
$$
\prod\limits_{i=1}^{n}\frac{1-t_i}{t_i}=-1
$$
is equivalent to the condition that the points $P_{n-3}$, $A_{n-1}$,
and $A_n$ are collinear $($see Figure~\ref{pentagon}.II$)$.
\end{proposition}

In the case $n=3$ we have: {\it the points $A_{1}$, $A_{2}$, and
$A_3$ are contained in one line} $($Menelaus configuration$)$.

\begin{proof}
We prove the proposition by induction on $n$. By Menelaus' theorem
and Theorem~\ref{relations} the statement holds for all
configurations with $n=3$. The induction step is the same as in
the proof of Proposition~\ref{geom2}, so we omit it here.
\end{proof}

\begin{remark}
The condition of Corollary~\ref{flex} can be rewritten as
Geometric condition I in the case of even $n$, and as Geometric
condition II in the case of odd $n$.
\end{remark}

\begin{corollary}
Any Kokotsakis mesh with a triangle in the middle, in general
position, is not infinitesimal flexible $($and therefore it is not
flexible$)$. \qed
\end{corollary}

In particular we have the following conclusion. Consider a closed
mesh in general position (i.e. there are no 3 vertices lying in
one line) whose edge-graph is 4-valent. Then the mesh is not
flexible if at least one of its faces is a triangle. Actually this
means that the Euler characteristic of the surface is
non-positive. So a flexible mesh in general position is not
homeomorphic to the sphere or to the projective plane.


\section{Flexibility conditions}\label{Sec2}

\subsection{Semialgebraicity of the set of flexible Kokotsakis meshes}

First we define a configuration space of all Kokotsakis meshes
whose central face is an $n$-gon. Any Kokotsakis mesh contains
$3n$ points. So the configuration space of all possible vertex
configurations is $(\r^3)^{3n}$. We will not consider
configurations for which some face has three consequent vertices
in a line, or which has faces with coinciding vertices. Denote
this set by $\Sigma_n$. Finally, the {\it configuration space of
all Kokotsakis meshes} is by definition
$(\r^3)^{3n}\setminus\Sigma_n$, denote it by $\Theta_n$.

Notice that we are not restricted to the case of planar faces in
this section. Consider the vector field:
$$
\begin{array}{lll}
\dot{v_i}&=&\displaystyle \left(\prod\limits_{j=1}^{i-1}
\frac{(a_{j-1},v_j,w_j)}{(a_j,v_j,w_j)}\right)[a_{i-1},v_i]\quad \hbox{and } \\
\dot{w_i}&=&\displaystyle \left(\prod\limits_{j=1}^{i}
\frac{(a_{j-1},v_j,w_j)}{(a_j,v_j,w_j)}\right)[a_i,w_i]\quad \hbox{for } i=1,\ldots, n.\\
\end{array}
$$
and denote it by $\xi_n$. This vector field is clearly everywhere
defined in $\Theta_n$.

Further when, in this section,  we speak of a motion we assume that
the special face in the middle is fixed while the others are
moving.

\begin{theorem}
Consider a flexible mesh in $\Theta_n$. The motion trajectory of
the mesh in some small neighborhood is contained in some integral
curve of the vector field $\xi_n$.
\end{theorem}

\begin{proof}
By Proposition~\ref{basis} a tangent vector for the trajectory of
a flexible mesh $M$ in $\Theta_n$ is proportional to the vector
$\xi(M)$. This implies the statement of the theorem.
\end{proof}

Denote the set of all flexible meshes by $F_{\Theta_n}$. Let us
consider an infinite sequence $f_i$ of algebraic functions on
$(\r^3)^{3n}$, defined inductively:
\\
{\it i$)$.} Take $f_1$ to be the numerator of the rational
function $\chi(M)-1$.
\\
{\it ii$)$.} Take $f_n$ to be the numerator of the rational
function $\langle\grad(f_{n-1}),\xi_n\rangle$.

\begin{remark}
We ignore denominators since they are nonzero algebraic functions
on $\Theta_n$.
\end{remark}

\begin{theorem}
The closure of $F_{\Theta_n}$ in $\Theta_n$ is a semialgebraic set
coinciding with
$$
\bigcap\limits_{i=1}^{+\infty} \{f_i=0\}\cap \Theta_n.
$$
\end{theorem}

\begin{proof}
Notice that the vector field $\xi(M)$ is non-zero at all the
points of $\Theta_n$.

By Proposition~\ref{basis}, for a flexible mesh $M$ not contained
in $\Sigma_n$ the velocity vector is proportional to the vector
$\xi(M)$, and therefore satisfies all the conditions $f_n=0$.

On the other hand all functions $f_n$ are algebraic and
therefore their common locus is an algebraic variety (by Hilbert's
basis theorem it is sufficient to take only a finite number of such
equations). Therefore in a nonsingular point of this variety we
have all the existence conditions for a finite motion along the
field $\xi(M)$ (see in~\cite{Arn}). Since the set of singularities
of an algebraic variety is nowhere dense in this variety, the
closure of the set of nonsingular points of a variety is the
variety itself.
\end{proof}

\subsection{Vector calculus of flexibility conditions}

In practice it is quite hard to calculate the above $f_n$
explicitly. To check some given configuration we recommend a
slightly different technique.

Since $\chi(M)=0$ along any trajectory, we immediately have many
conditions of flexibility: $\chi^{(n)}(M)=0$, where the derivatives
are taken with respect of the time parameter of the vector field.
We recommend to simplify the obtained function after each round of differentiation
using previous ones.

Let us introduce basic ingredients to calculate derivatives. We use a prime to indicate differentiation.

\begin{longtable}{l}
$
\langle a_{i},v_i\rangle'=-(a_{i-1},a_{i},v_i)
\prod\limits_{j=1}^{i-1} \frac{(a_{j-1},v_j,w_j)}{(a_j,v_j,w_j)};$\\$
\langle a_{i-1},w_i\rangle'=(a_{i-1},a_i,w_i)
\prod\limits_{j=1}^{i} \frac{(a_{j-1},v_j,w_j)}{(a_j,v_j,w_j)};$\\$
(a_{i-1},a_i,v_i)'= \Big(\langle a_{i-1},a_{i-1}\rangle \langle
a_{i},v_i\rangle - \langle a_{i-1},a_i\rangle \langle
a_{i-1},v_i\rangle   \Big)
\prod\limits_{j=1}^{i-1}
\frac{(a_{j-1},v_j,w_j)}{(a_j,v_j,w_j)};$\\$
(a_{i-1},a_i,w_i)'= \Big(\langle a_{i-1},a_{i}\rangle\langle
a_{i},w_i\rangle - \langle a_i,a_{i}\rangle \langle
a_{i-1},w_i\rangle \Big) \prod\limits_{j=1}^{i}
\frac{(a_{j-1},v_j,w_j)}{(a_j,v_j,w_j)};$\\$
\begin{aligned}
 (a_k,v_i,w_i)' = &
\begin{array}{c}
\Big( \langle a_k, v_i\rangle \langle a_{i-1},w_i\rangle- \langle
a_{i-1},a_k\rangle \langle v_i,w_i\rangle \Big)
\prod\limits_{j=1}^{i-1} \frac{(a_{j-1},v_j,w_j)}{(a_j,v_j,w_j)}
+\\
\end{array}
\\
&
\begin{array}{l}
\Big( \langle a_{i},a_{k}\rangle \langle v_i,w_{i}\rangle-\langle
a_i,v_i\rangle \langle a_k,w_{i}\rangle\Big)
\prod\limits_{j=1}^{i} \frac{(a_{j-1},v_j,w_j)}{(a_j,v_j,w_j)}.
\end{array}
\end{aligned}
$%
\end{longtable}

\noindent
In fact, to calculate $\chi^{(n)}(M)$, one only needs the cases
$k=i{-}1$ and $k=i$ for the last formula.

\begin{example}
The second condition $(\chi(M))'=0$ is equivalent to

\begin{longtable}{l}
$
\displaystyle
\sum\limits_{i=1}^{n}\left(
\frac{
 \langle a_i, v_i\rangle \langle a_{i-1},w_i\rangle-
\langle a_{i-1},a_i\rangle \langle v_i,w_i\rangle
}{(a_i,v_i,w_i)}
\prod\limits_{j=1}^{i-1}\frac{(a_{j-1},v_j,w_j)}{(a_i,v_i,w_i)}
\right)+
$
\\
$
\displaystyle
\sum\limits_{i=1}^{n}\left( \frac{ \langle a_{i},a_{i}\rangle
\langle v_i,w_{i}\rangle-\langle a_i,v_i\rangle \langle
a_i,w_{i}\rangle}{(a_i,v_i,w_i)}
\prod\limits_{j=1}^{i} \frac{(a_{j-1},v_j,w_j)}{(a_j,v_j,w_j)}
\right)
-
$
\\
$%
\displaystyle
\sum\limits_{i=1}^{n}\left(
\frac{
\langle a_{i-1}, v_i\rangle \langle a_{i-1},w_i\rangle-
\langle a_{i-1},a_{i-1}\rangle \langle v_i,w_i\rangle
}{(a_{i-1},v_i,w_i)}
\prod\limits_{j=1}^{i-1}
\frac{(a_{j-1},v_j,w_j)}{(a_j,v_j,w_j)}\right)
-
$
\\
$
\displaystyle
\sum\limits_{i=1}^{n}\left(
\frac{
 \langle a_{i-1},a_{i}\rangle \langle
v_i,w_{i}\rangle-\langle a_i,v_i\rangle \langle
a_{i-1},w_{i}\rangle}{(a_{i-1},v_i,w_i)}
\prod\limits_{j=1}^{i}
\frac{(a_{j-1},v_j,w_j)}{(a_j,v_j,w_j)}\right)=0.
$
\\
\end{longtable}
\end{example}

\section{Internal flexibility conditions for ($3{\times}3$)-meshes}\label{Sec3}
In this section we describe a technique to verify the flexibility
of a ($3{\times}3$)-mesh with planar faces in terms of inner
geometry of a mesh, i.e.,~in terms of the fixed angles between the
edges of the same faces.

 \vspace{2mm}

Consider a $(3{\times} 3)$-mesh with a planar quadrangle
$A_1A_2A_3A_4$ in the middle. Denote the angles between the faces
adjacent to $A_iA_{i+1}$ by $\omega_i$ ($i=1,\ldots, n$).

\begin{lemma}
The following relation holds for any vertex $A_i$
$$
\begin{array}{l}
\cos\alpha_i\cos\beta_i\cos\varphi_i+
\sin\alpha_i\sin\beta_i\cos\varphi_i\cos\omega_{i-1}+
\sin\alpha_i\cos\beta_i\sin\varphi_i\cos\omega_{i}-\\
\cos\alpha_i\sin\beta_i\sin\varphi_i\cos\omega_{i-1}\cos\omega_i+
\sin\beta_i\sin\varphi_i\sin\omega_{i-1}\sin\omega_i=\cos
\gamma_i.
\end{array}
$$
Denote this equation by $(R_i)$.
\end{lemma}

\begin{proof}
Note that the values of the angles do not depend on the lengths of
vectors $a_i$, $a_{i-1}$, $v_i$, $w_i$. So we can choose them equal $1$.
Let us make calculations in the coordinate system where
$a_i=(1,0,0)$, and $a_{i-1}=(\cos\alpha, \sin\alpha,0)$. Then we
have
$$
\begin{array}{l}
v_i=(\cos\varphi_i, \sin\varphi_i\cos\omega_i,
\sin\varphi\sin\omega_i),\\
w_i=\cos\beta_i(\cos\alpha_i,\sin\alpha_i,0)+
\sin\beta_i\Big(\cos\omega_{i-1}(\sin\alpha_i,-\cos\alpha_i,0)+\sin\omega_{i-1}(0,0,1)\Big)
\end{array}
$$
The formula of the lemma coincides with $\langle
v_i,w_i\rangle=\cos\gamma_i$ and therefore holds.
\end{proof}

We get the desired equations in 4 steps.

\vspace{2mm}

{\it Step 1.} Denote $\cos \omega_i$ by $t_i$. The first step is
to get rid of $\sin\omega_i\sin\omega_{i-1}$ in formula $(R_i)$.
This can be done by putting the corresponding term to the right
and the rest to the left, than squaring both sides, and finally substituting
$\sin^2\omega_i\sin^2\omega_{i-1}=(1-t_i^2)(1-t_{i-1}^2)$. Denote
the resulting equation by $(\bar{R}_i)$. It leads to
$$
C_{1,i}+C_{2,i}t_{i-1}+C_{3,i}t_{i}+C_{4,i}t_{i-1}^2+
C_{5,i}t_{i-1}t_{i}+C_{6,i}t_{i}^2+C_{7,i}t_{i-1}^2t_{i}+C_{8,i}t_{i-1}t_{i}^2+C_{9,i}t_{i-1}^2t_{i}^2=0,
$$
where
$$
\begin{array}{lll}
C_{1,i}=(\cos\alpha_i\cos\beta_i\cos\varphi_i-\cos\gamma_i)^2-\sin^2\beta_i\sin^2\varphi_i;\\
C_{2,i}=2(\cos\alpha_i\cos\beta_i\cos\varphi_i-\cos\gamma_i)\sin\alpha_i\sin\beta_i\cos\varphi_i;\\
C_{3,i}=2(\cos\alpha_i\cos\beta_i\cos\varphi_i-\cos\gamma_i)\sin\alpha_i\cos\beta_i\sin\varphi_i;\\
C_{4,i}=\sin^2\alpha_i\sin^2\beta_i\cos^2\varphi_i+\sin^2\beta_i\sin^2\varphi_i;\\
C_{5,i}=2(\cos\beta_i\cos\varphi_i-2\cos^2\alpha_i\cos\beta_i\cos\varphi_i
+\cos\alpha_i\cos\gamma_i)\sin\beta_i\sin\varphi_i;\\
C_{6,i}=\sin^2\alpha_i\cos^2\beta_i\sin^2\varphi_i+\sin^2\beta_i\sin^2\varphi_i;\\
C_{7,i}=-2\sin\alpha_i\cos\alpha_i\sin^2\beta_i\sin\varphi_i\cos\varphi_i;\\
C_{8,i}=-2\sin\alpha_i\cos\alpha_i\sin\beta_i\cos\beta_i\sin^2\varphi_i;\\
C_{9,i}=-\sin^2\alpha_i\sin^2\beta_i\sin^2\varphi_i.\\
\end{array}
$$

 {\it Step 2.} Now we recursively evaluate $t_i^2$ as a
linear function in $t_i$ whose coefficients are rational functions
of $t_{i-1}$. We will need only the formulae for $t_2^2$ and
$t_4^2$.

{\it Step 3.} We use $(\bar{R}_3)$ and the formula for $t_2^2$ given in Step~2
to evaluate $t_2$ as a function of $t_1$ and $t_3$,
and analogously we employ $(\bar{R}_4)$ with the formula for $t_4^2$ given in Step~2  to
evaluate $t_4$ as a function of $t_1$ and $t_3$. We get
$$
\begin{array}{l}
t_2=-\frac{(C_{9,2}t_{1}^2+C_{8,2}t_{1}+C_{6,2})(C_{6,3}t_{3}^2+C_{3,3}t_{3}+C_{1,3})-
(C_{9,3}t_{3}^2+C_{7,3}t_{3}+C_{4,3})(C_{4,2}t_{1}^2+C_{2,2}t_{1}+C_{1,2})}
{(C_{9,2}t_{1}^2+C_{8,2}t_{1}+C_{6,2})(C_{8,3}t_{3}^2+C_{5,3}t_{3}+C_{2,3})-
(C_{9,3}t_{3}^2+C_{7,3}t_{3}+C_{4,3})(C_{7,2}t_{1}^2+C_{5,2}t_{1}+C_{3,2})};\\
t_4=-\frac{(C_{9,1}t_{1}^2+C_{7,1}t_{1}+C_{4,1})(C_{4,4}t_{3}^2+C_{2,4}t_{3}+C_{1,4})-
(C_{9,4}t_{3}^2+C_{8,4}t_{3}+C_{6,4})(C_{6,1}t_{1}^2+C_{3,1}t_{1}+C_{1,1})}
{(C_{9,1}t_{1}^2+C_{7,1}t_{1}+C_{4,1})(C_{7,4}t_{3}^2+C_{5,4}t_{3}+C_{3,4})-
(C_{9,4}t_{3}^2+C_{8,4}t_{3}+C_{6,4})(C_{8,1}t_{1}^2+C_{5,1}t_{1}+C_{2,1})}.\\
\end{array}
$$

{\it Step 4.} Now for $t_1$ and $t_3$ we have the system
$$
\left\{\begin{array}{lcl}
R_1(t_1,t_4(t_1,t_3))&=&0\\
R_2(t_1,t_2(t_1,t_3))&=&0\\
\end{array}
\right. .
$$
First, note that the functions $R_1$ and $R_2$ are rational in
variables $t_1$ and $t_3$. So we can restrict ourselves to
consider the numerators. Denote these numerators by
$f_1(t_1,t_3)$, and $f_2(t_1,t_3)$. The functions $f_1$ and $f_2$
are polynomials of degree 10, and we need to check if these two
polynomials define a non-discrete set. This happens only if $f_1$
and $f_2$ have a common nonconstant factor. Take a resultant of
$f_1$ and $f_2$ considered as polynomials in $t_3$. This resultant
denote it by $p(t_1)$ is a polynomial in $t_1$.

\begin{theorem}
A mesh in $\Theta_n$ is flexible if and only if all the coefficients
of the polynomial $p(t_1)$ equal zero.
\end{theorem}


\begin{proof} We give an outline of the proof.

It is clear that if a mesh is flexible then the set of common solutions of
$R_i=0$ for $i=1,2,3,4$ is nondiscrete and therefore one of the
equations can be eliminated. By Proposition~\ref{basis} in the case
of meshes in $\Theta_n$ all the $t_i$'s are nonconstant while the mesh is
moving, so the resultant $p(t_1)$ is identically zero.

Suppose now that all the coefficients of the polynomial $p(t_1)$ equal zero,
then there are at most three independent equations among $R_i=0$ for $i=1,2,3,4$.
The common set of solutions of the equations $R_i=0$ is almost always at least one-dimensional.

The dimension of this set is zero only if the corresponding mesh
is a limit of a sequence of flexible meshes with possible motion
curves shrinking to this mesh.
Such mesh is infinitesimally flexible with more than one
degrees of freedom. By Proposition~\ref{basis} such meshes are not in $\Theta_n$.
\end{proof}

\begin{corollary}
Flexibility of meshes in $\Theta_n$ is determined by the shapes of
faces $($but not of an embedding in the space as a mesh$)$. \qed
\end{corollary}

To simplify symbolic calculations we recommend to compute the
resultants of $f_1(n,t_3)$ and $f_2(n,t_3)$ for $-50 \le n \le 50$.
If all 101 (=$10^2+1$) resultants evaluate to zero, than the polynomial
$p(t_1)$, which has at most 100 zeroes, has all zero coefficients.

\vspace{2mm}

Finally, we give an algorithm that calculates the functions
$f_1(x,y)$ and $f_2(x,y)$ (as $f[1]$ and $f[2]$ in the output
below).

\vspace{1mm}
\begin{quote}\footnotesize
\begin{verbatim}

TwoEquations:=proc(alpha,beta,gama,phi,N)  local i,C,R,t,f,tn,td;

for i from 1 to 4 do
  C[1,i]:=(cos(alpha[i])*cos(beta[i])*cos(phi[i])-cos(gama[i]))^2-
          sin(beta[i])^2*sin(phi[i])^2;
  ...
  # Here we write all the formulae for C[2,i],...,C[9,i] from Step 1
 end do:

tn[2]:=-((C[6,2]+C[8,2]*t[1]+C[9,2]*t[1]^2)*(C[1,3]+C[6,3]*t[3]^2+C[3,3]*t[3])-
        (C[7,3]*t[3]+C[4,3]+C[9,3]*t[3]^2)*(C[1,2]+C[2,2]*t[1]+C[4,2]*t[1]^2)):
td[2]:=((C[6,2]+C[8,2]*t[1]+C[9,2]*t[1]^2)*(C[8,3]*t[3]^2+C[5,3]*t[3]+C[2,3])-
       (C[7,3]*t[3]+C[4,3]+C[9,3]*t[3]^2)*(C[3,2]+C[5,2]*t[1]+C[7,2]*t[1]^2)):

tn[4]:=-((C[1,4]+C[2,4]*t[3]+C[4,4]*t[3]^2)*(C[4,1]+C[7,1]*t[1]+C[9,1]*t[1]^2)-
        (C[6,4]+C[8,4]*t[3]+C[9,4]*t[3]^2)*(C[1,1]+C[3,1]*t[1]+C[6,1]*t[1]^2)):
td[4]:=((C[3,4]+C[5,4]*t[3]+C[7,4]*t[3]^2)*(C[4,1]+C[7,1]*t[1]+C[9,1]*t[1]^2)-
       (C[6,4]+C[8,4]*t[3]+C[9,4]*t[3]^2)*(C[2,1]+C[5,1]*t[1]+C[8,1]*t[1]^2)):

RBar[1]:=(C[1,1]+C[6,1]*t[1]^2+C[3,1]*t[1])*td[4]^2+
      (C[8,1]*t[1]^2+C[5,1]*t[1]+C[2,1])*td[4]*tn[4]+
      (C[7,1]*t[1]+C[4,1]+C[9,1]*t[1]^2)*tn[4]^2:
RBar[2]:=(C[1,2]+C[2,2]*t[1]+C[4,2]*t[1]^2)*td[2]^2+
      (C[3,2]+C[5,2]*t[1]+C[7,2]*t[1]^2)*td[2]*tn[2]+
      (C[6,2]+C[8,2]*t[1]+C[9,2]*t[1]^2)*tn[2]^2:

f[1]:=(x,y)->(subs({t[1]=x,t[3]=y},simplify(RBar[1]))):
f[2]:=(x,y)->(subs({t[1]=x,t[3]=y},simplify(RBar[2]))):

return(f):
end proc;
\end{verbatim}
\end{quote}

Now we can compute resultants of $f[1](y_0,y)$ and $f[2](y_0,y)$
for any value $y_0$ with any precision $N$, or symbolically.

\vspace{1mm}
\begin{quote}\footnotesize
\begin{verbatim}
y0:=1;
N:=50;
f:=TwoEquations(alpha,beta,gama,phi,N):
ff[1]:=evalf[N](f[1](y0,y)):
ff[2]:=evalf[N](f[2](y0,y)):
evalf[N](resultant(ff[1]/coeftayl(ff[1],y=0,4),ff[2]/coeftayl(ff[2],y=0,4),y));
\end{verbatim}
\end{quote}

Still we do not know which of the listed 101 equations are
independent, the first interesting problem in this direction is
the following.

\section{Some questions for further study}

\subsection{Geometry of conditions}

In Section~\ref{Sec1} we studied the geometric conditions for infinitesimal
rigidity of meshes with planar faces, the following question arise here.

\begin{problem}
What is the geometric infinitesimal flexibility condition in case
of meshes with nonplanar faces?
\end{problem}
A similar problem remains open for the case of flexibility conditions
introduced in Section~\ref{Sec2}.
\begin{problem}
Find a geometric interpretation of the condition $\chi^{(n)}(M)=0$
for $n\ge 1$.
\end{problem}
We do not know how the condition $\chi'(M)=0$ is related to
the second-order rigidity condition defined by A.~I.~Bobenko,
et al. in~\cite{BHS}.

\subsection{Codimension calculation}

In Section~\ref{Sec2} we described the closure of all flexible meshes
$F_{\Theta_n}$ by algebraic equations. We conjecture that it is a complete
intersection. It is most probable that all the components of the
closure are of the same dimension.

\begin{remark}
Let us say a few words about the case of $(3{\times} 3)$-meshes.
Direct calculation of the tangent space in the points of
$F_{\Theta_4}$ based on the results of the second section shows that
the codimension of some of the irreducible components of
$F_{\Theta_4}$ is greater than or equal to~6. On the other hand, the
Voss surfaces (c.f.~\cite{Vos} and~\cite{BHS})
forms a set of flexible meshes of codimension~8. In general we
suspect that all the irreducible components of $F_{\Theta_4}$ has
the same codimension between~6 and~8. Note that for codimension~8
we know the following {\it families of flexible meshes}:

\vspace{2mm}

{\bf i)} Voss surfaces (defined by equations $\alpha_i=\gamma_i$,
$\beta_i=\varphi_i$ for $i=1,2,3,4$).

{\bf ii)} symmetric meshes where the angles at the vertices $A_1$
and $A_2$ coincide with the angles at the vertices $A_4$ and
$A_3$:
$$
\alpha_1=\alpha_4,\,\beta_1=\varphi_4,\,\gamma_1=\gamma_4,\,\varphi_1=\beta_4;\quad
\alpha_2=\alpha_3,\,\beta_2=\varphi_3,\,\gamma_2=\gamma_3,\,\varphi_2=\beta_3.
$$

{\bf iii)} symmetric meshes where the angles at the vertices $A_1$
and $A_4$ are the same as at the vertices $A_2$ and $A_3$.

{\bf iv)} meshes obtained from the cases i)---iii) by replacing some
$v_i$ and/or $w_j$ by $-v_i$ and/or $-w_j$. For instance, if we
replace $v_1$ by $-v_1$ in the case of Voss meshes, then we get
equations $\alpha_1+\gamma_1=\pi$, $\beta_1+\varphi_1=\pi$, and
$\alpha_i=\gamma_i$, $\beta_i=\varphi_i$ for $i=2,3,4$.

\vspace{2mm}

If the codimension of $F_{\Theta_4}$ is~8, then the natural problem
is {\it to complete the list of flexible meshes in general
position}.
\end{remark}

All this leads to the following problem.
\begin{problem}\label{prob3}
Find the codimension of the closure of $F_{\Theta_n}$ for $n=4,5,\ldots$
\end{problem}

The same question is interesting for the case of an algebraic set defined
by 101 equations in Section~\ref{Sec3}, {\it does this set has the same codimension?}

\vspace{5mm}

{\bf Acknowledgement.} The author is grateful to J.~Wallner for
constant attention to this work and Ch. M\"uller for useful remarks.
This work was performed at Technische Universit\"at Graz within
the framework of the project
``Computational Differential Geometry'' (FWF grant No.~S09209).

\vspace{.3cm}


\end{document}